\documentclass[10 pt]{amsart}

\usepackage{amsmath} 
\usepackage{amsfonts}
\usepackage{amssymb}
\usepackage{amstext}
\usepackage{amsbsy}
\usepackage{amsopn}
\usepackage{amsthm}
\usepackage{amsxtra}
\usepackage{graphicx}
\usepackage{caption}
\usepackage{subcaption}
\usepackage{color}
\usepackage{hyperref}
\usepackage{enumerate}
\usepackage{amscd}
\usepackage{tikz}

\newtheorem{theorem}{Theorem}[section]
\newtheorem{lemma}[theorem]{Lemma}

\newtheorem{proposition}[theorem]{Proposition}

\newtheorem*{conjecture*}{Conjecture}
\newtheorem*{corollary*}{Corollary}
\newtheorem*{claim*}{Claim}
\newtheorem*{theorem*}{Theorem}

\theoremstyle{remark}

\theoremstyle{definition}

\newcommand{\one}{\mathbf{1}}

\newcommand{\A}{\mathcal{A}}

\newcommand{\Z}{\mathbb{Z}}

\newcommand{\N}{\mathbb{N}}

\newcommand{\CA}{\mathcal{A}}

\newcommand{\CL}{\mathcal{L}}
\newcommand{\CW}{\mathcal{W}}

\newcounter{note}

\title{Realizing ergodic properties in zero entropy subshifts}
\author{Van Cyr}
\address{Bucknell University, Lewisburg, PA 17837 USA}
\email{van.cyr@bucknell.edu}
\author{Bryna Kra}
\address{Northwestern University, Evanston, IL 60208 USA}
\email{kra@math.northwestern.edu}

\subjclass[2010]{}


\thanks{The second author was partially supported by NSF grant 1800544.}

\begin{document}

\begin{abstract}

A subshift with linear block complexity has at most countably many ergodic measures, 
and we continue of the study of the relation between such complexity and the invariant measures. 
By constructing minimal subshifts whose block complexity is arbitrarily close to linear but has
uncountably many ergodic measures, 
we show that this behavior fails as soon as the block complexity is superlinear.  
With a different construction, we show that there exists a minimal subshift 
with an ergodic measure whose slow entropy grows slower than any given rate tending to 
infinitely but faster than any other rate majorizing this one yet still growing subexponentially.  
These constructions lead to obstructions in using subshifts in applications to properties of the prime numbers and in finding a measurable version of the complexity gap that arises for shifts of sublinear complexity. 
\end{abstract}

\maketitle

\section{Introduction}
Assume that $(X, \sigma)$ is a subshift over the finite alphabet $\A$, meaning that $X\subset \A^\Z$ is a closed set 
that is invariant under the left shift $\sigma\colon \CA^\Z\to\CA^\Z$.  
The block complexity $p_X(n)$ of the shift is defined to be the number of words of length $n$ which occur 
in any $x\in X$.  
Boshernitzan~\cite{B} showed that a minimal subshift with linear block complexity has only finitely many ergodic measures, 
where the number depends on the complexity growth.  
In~\cite{CK}, we showed that any subshift (minimal or not) with linear block complexity has at most finitely many 
nonatomic ergodic measures, and so at most countably many ergodic measures (with no requirement that the measures are nonatomic).  
In the same article, we give examples of subshifts with block complexity arbitrarily close to linear which have uncountably many nonatomic 
ergodic measures.  Our main result is to show 
there is no complexity bound beyond linear on a subshift that suffices for guaranteeing there are at most countably many ergodic measures.
More precisely, we show as soon as the growth is superlinear, we can have the maximal number of ergodic measures:
\begin{theorem}
\label{thm:main} 
If $(p_n)_{n\in\N}$ is a sequence of natural numbers such that 
$$ 
\liminf_{n\to\infty}\frac{p_n}{n}=\infty, 
$$ 
then there exists a minimal subshift $X$ which supports uncountably many ergodic measures and is such that 
$$ 
\liminf_{n\to\infty}\frac{P_X(n)}{p_n}=0. 
$$ 
\end{theorem} 

The distinction between countably and uncountably many ergodic measures supported by a subshift has recently received attention, as it plays a role in the deep results of Frantzikinakis and Host~\cite{FH} on the complexity of the Liouville shift. 
More precisely, by studying the subshift naturally associated to Liouville function $\lambda(n)$ (see Section~\ref{subshift-language})  and the number of ergodic measures it supports, they conclude that the Liouville function has superlinear complexity.  
Given the example we construct in Theorem~\ref{thm:main}, any approach to showing that the Liouville function has higher growth rates must rely on further properties of the shift beyond the cardinality of the set ergodic measures supported by subshifts whose complexity grows at a given rate. 
In a further development, using different methods McNamara~\cite{M} has shown that the Liouville function has at least quadratic complexity.

Katok and Thouvenot~\cite{KT} and Ferenczi~\cite{Fer} defined the slow entropy, as a measure theoretic invariant to capture 
a measurable version of the  (topological) complexity.  They defined two growth rates, $P_{T}^-(n)$
and $P_{T}^+(n)$ of a measure preserving system $(X, \mu, T)$ and whether each these growth rates is slower $\prec$ or faster $\succ$ than a given growth rate is a measure theoretic invariant for the system and computable using a generating partition (see Section~\ref{sec:measure-theoretic} for precise definitions).  
Using a different construction, we exhibit the freedom on growth rate of the slow entropy,  with a minimal subshift of zero topological entropy such that the slow entropy grows slower than some (arbitrarily slowly growing) given sequence while faster than another (arbitrarily quickly within the class of subexponentially growing) given sequence: 
\begin{theorem}\label{thm:slow-entropy} 
Assume $(a_n)_{n\in\N}$ and $(b_n)_{n\in\N}$ are two non-decreasing sequences of positive integers such that $\displaystyle{\lim_{n\to\infty} a_n=\infty}$, $\displaystyle{\lim_{n\to\infty}\frac{1}{n}\cdot\log(b_n)=0}$, and $a_n\leq b_n$ for all $n\in\N$.  There exists a minimal subshift $(X_{\infty},\sigma)$ of topological entropy zero and an ergodic measure $\mu$ supported on $X_{\infty}$ such that 
$$ 
P_{\sigma}^-(n)\prec(a_n) \text{ and } P_{\sigma}^+(n)\succ(b_n). 
$$ 
\end{theorem}

If we only consider the upper growth rate $P_T^+(n)$, related constructions are given in Katok and Thouvenot~\cite{KT} and Serafin~\cite{S} of a subshift with zero topological entropy and $P_T^+(n)$ growing arbitrarily quickly (but still subexponetially growing).  
The restrictions given in our theorem on the sequences $(a_n)_{n\in\N}$ and $(b_n)_{n\in\N}$ are the weakest for which a result like ours could hold, in a precise sense that we explain (see Section~\ref{sec:measure-theoretic}).
By a theorem of Ferenczi~\cite{Fer}, it follows from the lower bound on the growth  
$P_{\sigma}^+(n)\succ(b_n)$, that the system we construct is not a Kronecker system. 
From the upper bound $P_{\sigma}^-(n)\prec(a_n)$, 
it follows that there is no sequence $(a_n)$ increasing to infinity 
and such that the analogous bound would give 
$P_{\sigma}^-(n)\prec(c_n)$ for all sequences $(c_n)$ 
increasing to infinity.  
This exhibits a different behavior than what happens in the topological setting.  
Namely, the Morse-Hedlund Theorem states that if there exists some $n\in\N$ such that 
$P_X(n)\leq n$ for some system system $(X, \sigma)$, then we have that the topological complexity function $P_X(n)$
is actually bounded for all $n\in\N$.  
Our construction shows that the measurable analog of the result fails. 
Since $P_{\sigma}^+(n)\succ\{b_n\}$, by an approximation argument we have that the topological complexity 
$P_X(n) \geq b_n$ for infinitely many $n$.  
In particular, there is no subexponentially growing sequence $(b_n)_{n\in\N}$ 
such that  any  subshift whose word complexity exceeds $b_n$ for infinitely many $n$ 
must have positive entropy.  
This can be viewed as a counterexample to a dual version 
of the  Morse-Hedlund Theorem, 
meaning there exists no subexponentially growing sequence that asymptotically bounds the complexity of every zero entropy subshift.

\section{Background and notation}

\subsection{Symbolic Dynamics} 
Assume that $\A$ is a finite set endowed with the discrete topology, and we call $\A$ the {\em alphabet}.   The 
space $\A^{\Z}$, endowed with the product topology, is a compact, metrizable space.  
An element $x\in \A^{\Z}$ denotes a bi-infinite sequence in the alphabet $\A$, 
meaning that $x=(x_i)_{i\in\Z}$ with each $x_i\in\A$. 
The left shift $\sigma\colon\A^{\Z}\to\A^{\Z}$ defined by $(\sigma x)_i:=x_{i+1}$ is continuous and the dynamical system $(\A^{\Z},\sigma)$ is the {\em full $\A$-shift}.  A {\em subshift} $X\in\A^{\Z}$ is the restriction of $\sigma$ to any closed, $\sigma$-invariant set $X$. 

\subsection{Words and complexity} If $w=(w_0,\dots,w_{n-1})\in\A^n$, 
the {\em cylinder set} $[w]$ in $\A^{\Z}$ determined by $w$ is defined to be
$$ 
[w]:=\{x\in\A^{\Z}\colon x_i=w_i\text{ for all }0\leq i\leq n-1\}. 
$$ 
If $X\subseteq\A^{\Z}$ is a subshift, then the {\em language $\mathcal{L}(X)$ of $X$} is the set of all words $w\in\A^*$ such that $[w]\cap X\neq\emptyset$.  For $n\in\N$, the set $\mathcal{L}_n(X)$ of {\em words of length $n$} in $X$ is the set 
$$ 
\mathcal{L}_n(X):=\{w\in\A^n\colon[w]\cap X\neq\emptyset\}. 
$$ 
The {\em block complexity} (also known as the {\em word complexity function}) $p_X\colon\N\to\N$ of $X$ is the function that counts the number of words of each length in $X$.  Thus 
$$ 
p_X(n):=|\mathcal{L}_n(X)|. 
$$ 

\subsection{The natural subshift associated to a language}
Given a subshift $(X, \sigma$), it's language $\CL(X)$ satisfies the properties that: 
\begin{enumerate}
\item
\label{it:one} If $w\in \CL(X)$, then every subword of $w$ also belongs to $\CL(X)$.
\item 
\label{it:two}
If $w\in\CL(X)$, then there exist nonempty words $u,v\in \CL(X)$ such that $uwv\in\CL(X)$.  
\end{enumerate}
Conversely, if $\CL$ is any collection of words over an alphabet $\CA$ satisfying condition~\eqref{it:one}, then $\CL = \CL(X)$ for some subshift $(X, \sigma)$. Thus the language of a subshift determines the subshift, and we can define a subshift by specifying its language.  When we do so, we say that $(X, \sigma)$ is the subshift defined by the language $\CL$.

Moreover, given a finite list of words $v_1, \ldots, v_k$, we can consider the collection of all  bi-infinite concatenations of these words, 
and by taking the shifts of these bi-infinite sequences, we obtain a subshift $(X, \sigma)$.  We refer to this as the
 {\em subshift defined by the words $v_1, \ldots, v_k$}.

\subsection{The natural subshift associated to a sequence}\label{subshift-language}  
Suppose $h\colon\N\to\A$ is a function.  Fix some $a\in\A$ and define 
$$ 
y:=\begin{cases}
y_i=h(i) & \text{ if } i>0; \\ 
y_i=a & \text{ if } i\leq 0. 
\end{cases}
$$ 
Then the set 
$$ 
Y_h:=\overline{\{\sigma^ny\colon n\in\Z\}} 
$$ 
is a subshift.  Note that $Y_h$ contains at most $n$ words of length $n$ that are not found in contiguous blocks of the function $h\colon\N\to\A$.  Therefore number of such words differs from $P_{Y_h}(n)$ by at most $n$.  A word $w\in \mathcal{L}(Y)$ if and only if there are arbitrarily large $m\in\N$ such that $w_i=h(m+i)$ for all $0\leq i<|w|$.  The resulting shift $(Y_h, \sigma)$ is {\em transitive}, meaning that it has a point with dense orbit under the shift (note that the point $y$ has dense orbit). 

Of particular interest are functions $h$ that arise in number theory.  
For example, we can consider $h$ to be the Liouville function $\lambda(n)$, 
the sequence in the alphabet $\{-1, 1\}$ with $\lambda(n) = 1$ if $n$ has an even number of prime factors counted with multiplicity and otherwise is $-1$, or  take $h$ to be the M\"obius function $\mu(n)$, the sequence in the alphabet $\{-1, 0, 1\}$ with 
$\mu(n)=1$ if $n$ is square free and has an even number of prime factors, $\mu(n) = -1$  if $n$ is square free and has an odd number of prime factors, and otherwise $\mu(n) = 0$.  
Then studying the language of $Y_h$ gives insight into the number theoretical properties of $h$. 

\subsection{Invariant measures on certain subshifts} 
A Borel measure $\nu$ supported in $X$ is {\em invariant} if $\nu(A)=\nu(\sigma^{-1}A)$ for all Borel sets $A\subseteq X$ 
and is {\em ergodic} if $\nu(A)\cdot\nu(X\setminus A)=0$ whenever $A=\sigma^{-1}A$.  
It was recently shown~\cite{FH} that if $(Y,T)$ is a topological dynamical system of entropy zero and if $Y$ supports at most countably many ergodic measures, then $(Y,T)$ satisfies a logarithmic variant of the M\"obius Disjointness Conjecture of Sarnak.  Namely, under these conditions, for every $y\in Y$ and every $f\in C(Y)$ we have 
$$ 
\lim_{N\to\infty}\frac{1}{\log N}\sum_{n=1}^N\frac{f(T^ny)\mu(n)}{n}=0  
$$ 
where $\mu$ is the M\"obius function (and the same conclusion holds if $\mu$ is replaced with the Liouville function $\lambda$).  This result was leveraged in~\cite{FH} to give a lower bound on the growth of words 
in the language of the {\em Liouville shift}: $Y_{\lambda}$, 
showing that the Liouville shift has superlinear block growth (note again that the question of whether $P_{Y_{\lambda}}(n)$ grows superlinearly is equivalent to the question of whether the number of words of length $n$ that occur in the Liouville sequence grows superlinearly).  Along with the deep result on logarithmic disjointness that they prove, they rely on a bound on the number of ergodic measures that can be supported by a subshift with linear growth.  

As a corollary of our result in Theorem~\ref{thm:main}, any proof showing that the Liouville shift has a growth rate that is faster than some explicit superlinear bound necessarily relies on deeper information from number theory, rather than only on estimates on the number of ergodic measures. 
\section{Proof of Theorem~\ref{thm:main}} 
Suppose $(p_n)$ is a sequence of natural numbers satisfying 
$$ 
\liminf_{n\to\infty}\frac{p_n}{n}=\infty. 
$$ 
We construct a minimal subshift $X$, depending on the sequence $(p_n)$, 
that supports uncountably many ergodic measures and is such that 
$$ 
\limsup_{n\to\infty}\frac{P_X(n)}{p_n}=0. 
$$ 

We build the system inductively, by constructing words at each level and then using these words to define the language of a subshift.  Throughout we use superscripts to denote the level of the construction and subscripts to denote the words constructed at this level.  
\subsection*{Step 1: the base step of the construction}
 Set $\A=\{0,1\}$.  Define 
$$ 
w_1^1=\underbrace{00\cdots00}_{N_1\text{ times}}\hspace{-0.02 in}1 
$$ 
and 
$$ 
w_2^1=0\hspace{-0.02 in}\underbrace{11\cdots11}_{N_1\text{ times}}, 
$$ 
where $N_1$ is a large integer to be determined later.  
Note that the words $w_1^1$ and $w_2^1$ are distinct, as can be seen from the different frequencies of $0$'s and $1$'s. 
Define $X_1\subseteq\A^{\Z}$ to be the subshift of $\A^{\Z}$ consisting of all $x\in\A^{\Z}$ that can be written as bi-infinite concatenations of the words $w_1^1$ and $w_2^1$.  
Consider $n_1=\lfloor N_1/2\rfloor$ (note that $\lfloor \cdot\rfloor$ denotes the floor function).  
Then $\mathcal{L}_{n_1}(X_1)$ contains a word of all $0$'s, a word of all $1$'s, all words with exactly one $1$ (by 
concatenating $w_1^1$ with itself), all words with exactly one $0$ 
(by concatenating $w_2^1$ with itself), all words that begin with a continuous string of $0$'s followed by a continuous string of $1$'s 
(by concatenating $w_1^1$ and $w_2^1$), and all words with a continuous string of $1$'s followed by a continuous string of $0$'s (by concatenating $w_2^1$ and $w_1^1$).  
Thus for $n_1=\lfloor N_1/2\rfloor$, 
$$ 
P_{X_1}(n_1)=4n_1-2 
$$ 
and so if $Y\subseteq X_1$ is any subshift, then $P_Y(n_1)\leq4n_1-2$. 

\subsection*{Step 2: the inductive step} 
Assume we have constructed a nested sequence of subshifts 
$$ 
X_k\subseteq X_{k-1}\subseteq X_{k-2}\subseteq\cdots\subseteq X_1\subseteq\A^{\Z}
$$ 
and an increasing sequence of integers $n_1<n_2<\ldots<n_k$ such that we have $P_{X_i}(n_i)\leq\binom{n_{i-1}}{2}\cdot2^{i-1}|w_1^{i-1}|\cdot n_i$ for $i=1,2,\dots,k$. 
 Moreover, for each $i=1,2,\dots,k$, suppose we have constructed distinct words    

$$ 
w_1^i, w_2^i,\dots,w_{2^i}^i, 
$$ 
all of which have the same length, all of which lie in $\mathcal{L}(X_{i-1})$, all of which are concatenations of words from the set $\{w_1^{i-1},\dots,w_{2^{i-1}}^{i-1}\}$, and are such that $X_i$ is the subshift of $X_{i-1}$ consisting of all words that can be written as bi-infinite concatenations of $w_1^i,w_2^i,\dots,w_{2^i}^i$.  Further assume that for any $1\leq j<k\leq2^i$, there is no subword of length $|w_1^i|$ that occurs in both $w_j^iw_j^i$ and $w_k^iw_k^i$.  Note that since $w_1^i,\dots,w_{2^i}^i$ are concatenations of words from the set $\{w_1^{i-1},\dots,w_{2^{i-1}}^{i-1}\}$, 
all concatenations of $w_1^i,\dots,w_{2^i}^i$ are elements of $X_{i-1}$.  Finally, for $i>1$, suppose that for $1\leq j\leq 2^i$, the word $w_j^i$ contains each of the words $w_1^{i-1}, w_2^{i-1},\dots,w_{2^{i-1}}^{i-1}$ somewhere as a subword. 

To construct $X_{k+1}$, we start by defining
$$ 
w_1^{k+1}=\underbrace{(\overbrace{w_1^kw_1^k\cdots w_1^k}^{S_k\text{ times}}w_2^k)(\overbrace{w_1^kw_1^k\cdots w_1^k}^{S_k\text{ times}}w_2^k)\cdots(\overbrace{w_1^kw_1^k\cdots w_1^k}^{S_k\text{ times}}w_2^k)}_{N_k\text{ times}}w_2^kw_3^k\cdots w_{2^k}^k 
$$ 
and  
$$ 
w_2^{k+1}=\underbrace{(\overbrace{w_1^kw_1^k\cdots w_1^k}^{S_k\text{ times}}w_1^k)(\overbrace{w_1^kw_1^k\cdots w_1^k}^{S_k\text{ times}}w_1^k)\cdots(\overbrace{w_1^kw_1^k\cdots w_1^k}^{S_k\text{ times}}w_1^k)}_{N_k\text{ times}}w_2^kw_3^k\cdots w_{2^k}^k 
$$ 
where the brackets have no mathematical meaning other than to draw attention to the periodic nature of the initial prefix of these words and the fact that the word being periodized in $w_1^{k+1}$ is different from that in $w_2^{k+1}$.   Again, $N_k$ and $S_k$ are large integers to be determined later.  We continue the construction: for $1<j\leq 2^k$, 
defining
$$
v_j^k = \overbrace{w_j^kw_j^k\cdots w_j^k}^{S_k\text{ times}},
$$ we define 
$$ 
w_{2j-1}^{k+1}=w_1^kw_2^k\cdots w_{j-1}^k\underbrace{(v_j^kw_{j+1}^k)(v_j^kw_{j+1}^k)
\cdots (v_j^kw_{j+1}^k)}_{N_k\text{ times}}w_{j+1}^kw_{j+2}^k\cdots w_{2^k}^k
$$ 
and
$$ 
w_{2j}^{k+1}=w_1^kw_2^k\cdots w_{j-1}^k\underbrace{(v_j^kw_j^k)(v_j^kw_j^k)\cdots (v_j^kw_j^k)}_{N_k\text{ times}}w_{j+1}^k\cdots w_{2^k}^k, 
$$ 
where $w_{2^{k+1}+1}^{k+1}:=w_1^{k+1}$, meaning that subscripts are understood modulo $2^{k+1}$.

Note that for $1\leq j<i\leq2^{k+1}$, we claim that the words of length $|w_1^{k+1}|$ that occur in the word $w_i^{k+1}w_i^{k+1}$ are distinct from those that occur in the word $w_j^{k+1}w_j^{k+1}$ (provided $N_k$ is sufficiently large).  
If $i=j+1$, the claim follows because $w_{j+1}^k$ occurs at least $N_k$ times in any subword of length $|w_1^{k+1}|$ in $w_j^{k+1}w_j^{k+1}$ (once in each copy of the periodized word $v_j^kw_{j+1}^k$) and occurs at most $|w_1^k|\cdot2^k$ times in $w_i^{k+1}w_i^{k+1}$ (since it does not occur anywhere in the periodized word and can only occur in the prefix or suffix, which collectively have length $|w_1^k|\cdot2^k$).  For $N_k>|w_1^k|\cdot2^k$, no such word can occur in both $w_j^{k+1}w_j^{k+1}$ and $w_i^{k+1}w_i^{k+1}$.  If $i>j+1$, the claim follows because $w_j^k$ occurs at least $S_k\cdot N_k$ times in any subword of length $|w_1^k|$ in $w_j^{k+1}w_j^{k+1}$ and occurs at most $|w_1^k|\cdot2^k$ times in $w_i^{k+1}w_i^{k+1}$.  Further note that the frequency with which words of length $|w_j^k|$ that occur in $w_j^kw_j^k$, occur as subwords in $w_{2j-1}^{k+1}$ and $w_{2j}^{k+1}$ is at least $N_kS_k|w_1^k|/|w_1^{k+1}|$.  
By choosing $N_k$ sufficiently large, this can be made arbitrarily close to $1$.  Thus, 
by choosing $N_k$ sufficiently large, 
any word (or collection of words) that occur with frequency $\delta$ in $w_j^kw_j^k$ can be made to occur with frequency arbitrarily close to $\delta$ in $w_{2j-1}^{k+1}$ and $w_{2j}^{k+1}$.  
Furthermore, for $i\notin\{2j-1,2j\}$, the frequency with which words of length $|w_j^k|$ that occur in $w_j^kw_j^k$ occur in $w_i^{k+1}$ is at most $1-N_kS_k|w_1^k|/|w_1^{k+1}|$, 
as these words do not occur in any $w_t^kw_t^k$ for any $t\neq j$ (and subwords of this form occur with frequency at least $N_kS_k|w_1^k|/|w_1^{k+1}|$ in $w_i^{k+1}$).  Again, by choosing $N_k$ sufficiently large, 
this frequency can be made arbitrarily close to zero. 

Define $X_{k+1}\subseteq X_k$ to be the subshift consisting of all words that can be written as bi-infinite concatenations of the words $w_1^{k+1}, w_2^{k+1},\dots,w_{2^{k+1}}^{k+1}$.  Note that, by construction, every word of the form $w_j^{k+1}$ contains each of the words $w_1^k, w_2^k,\dots,w_{2^k}^k$ as a subword.  Furthermore $X_{k+1}\subseteq X_k$ and each element of $X_{k+1}$ can be written as a bi-infinite concatenation of the words $w_1^{k+1},w_2^{k+1},\dots,w_{2^{k+1}}^{k+1}$.  

Define $$n_{k+1}:=\lfloor|w_1^{k+1}|/2\rfloor$$ to be half the (common) length of the words $w_1^{k+1},\dots,w_{2^{k+1}}^{k+1}$.  We claim that $P_{X_{k+1}}(n_{k+1})\leq\left(\binom{2^k}{2}+1\right)n_{k+1}$.  Each of the words $w_i^{k+1}$ consists of a {\em prefix region} in which  $w_1^k,\dots,w_{i-1}^k$ are concatenated in order, then a {\em periodic region}
in which either $w_i^k$ is self-concatenated or $w_i^kw_i^k\cdots w_i^kw_{i+1}^k$ is self-concatenated, and finally a {\em suffix region} in which $w_{i+1}^k,\dots,w_{2^k}^k$ are concatenated in order.  Any word in $\mathcal{L}_{n_{k+1}}(X_{k+1})$ occurs either entirely within some word $w_i^{k+1}$ or partially overlaps two words $w_{i_1}^{k+1}w_{i_2}^{k+1}$.  For words of the first type, they may occur entirely within the periodic region or they start within the first $2^k|w_1^k|$ letters or they end within the last $2^k|w_1^k|$ letters of $w_i^{k+1}$.  
For those in the periodic region,  there are at most $2^{k+1}(S_k+1)|w_1^k|$ many such words since this is the number of words constructed multiplied by the maximal period of the periodic region.  For those overlapping two of the regions, 
there are at most $2^{k+1}|w_1^k|$ such words.  
Words of the second type must overlap $w_{i_1}^{k+1}$ and $w_{i_2}^{k+1}$ for some $i_1\neq i_2$ and so must start within the last $n_{k+1}$ letters of $w_{i_1}^{k+1}$, 
and there are at most $\binom{2^{k+1}}{2}n_{k+1}$ such words.  Therefore 
$$ 
P_{X_{k+1}}(n_{k+1})\leq2^{k+1}(S_k+1)|w_1^k|+2^{k+1}|w_1^k|+\binom{2^{k+1}}{2}n_{k+1}\leq\left(\binom{2^{k+1}}{2}+1\right)n_{k+1}, 
$$ 
so long as $n_{k+1}$ is sufficiently large when compared to $n_k$. 

This establishes the assumptions of the inductive hypothesis, and giving us an infinite nested chain of subshifts 
$$ 
\A^{\Z}\supseteq X_1\supseteq X_2\supseteq X_3\supseteq\cdots\supseteq X_k\supseteq\cdots 
$$ 

\subsection*{Step 3: construction and growth properties of the subshift $X$}  
The word $w_1^{k+1}$ starts with the word $w_1^k$ for all $k$ and so there is a $\{0,1\}$-coloring of $\N$ such that for all $k$, the prefix of length $|w_1^k|$ is $w_1^k$.  Let $\tilde{X}$ be the orbit closure of this word in $\A^{\N}$ and let $(X,\sigma)$ be its natural extension to a subshift of $\A^{\Z}$.  It follows immediately from the construction that 
$$ 
X\subseteq\bigcap_{k=1}^{\infty} X_k
$$ 
and $(X, \sigma)$ is a nonempty subshift.  
We have constructed an increasing sequence $n_1<n_2<\cdots<n_k<\cdots$ such that 
$$ 
P_X(n_k)\leq\left(\binom{2^k}{2}+1\right)n_k 
$$ 
for all $k>1$.  The integers $n_k$ are on the order of $1/2$ of the parameters $N_k|w_1^k|$, and 
in particular tend to infinity as $N_k$ tends to infinity.  The parameters $N_k$ have not yet been fixed, and we put some constraints on them now.  Recall that we are given the sequence $(p_n)$ such that
$$ 
\liminf_{n\to\infty}\frac{p_n}{n}=\infty. 
$$ 
For each $k\geq1$, there exists $M_k$ such that for all $n\geq M_k$, 
we have $$p_n>k\cdot\left(\binom{2^k}{2}+1\right)n.$$
Fix an increasing sequence of integers $(M_k)$ with this property.  We 
assume that $N_k$ is chosen to be sufficiently large such  that $N_k>2N_{k-1}$ and $n_k>M_k$.  
Since $P_X(n_k)/p_{n_k}<1/k$ for all $k$, we have that 
$$ 
\liminf_{n\to\infty}\frac{P_X(n)}{p_n}=0.
$$ 
Moreover, note that for any $k>1$ we have $X\subseteq X_k$.  Since $X_k$ consists only of words that can be obtained as bi-infinite concatenations of the words $w_1^k,\dots,w_{2^k}^k$ and 
each of these words contains all of the words $w_1^{k-1},\dots,w_{2^{k-1}}^{k-1}$ as a subword, 
it follows that each of the words $w_1^{k-1},\dots,w_{2^{k-1}}^{k-1}$ occurs syndetically in every element of $X_k$ (hence also in every element of $X$) with gap at most $|w_1^k|$.  Since every word in $\mathcal{L}(X)$ is itself a subword of $w_1^k$ for some $k$, it follows that every word in $\mathcal{L}(X)$ occurs syndetically in every element of $X$ with a bound on the gap that depends on the word but not on the element of $X$.  Hence $(X,\sigma)$ is minimal. 

\subsection*{Step 4: the set of ergodic measures on $X$} 
Fix a sequence $(\delta_i)$ of positive real numbers in the interval $(0,1)$ such that 
$$ 
\Delta:=\prod_{i=1}^{\infty}\delta_i>\frac{9}{10}. 
$$ 
Then for any $k>1$, we also have $\prod_{i=k}^{\infty}\delta_i>9/10$.  
Recall that that every word of the form $w_i^k$ consists of a prefix region, a periodic region, and a suffix region, where the lengths of the prefix and suffix regions are bounded independently of $N_k$.  
Thus we can further choose $N_k$ to grow sufficiently quickly such that 
$$ 
\frac{N_k}{|w_i^{k+1}|}\cdot\frac{S_k}{S_k+1}>\delta_k 
$$ 
for all $i=1,\dots,2^k$ (recall that all of these words have the same length).  It follows from this choice that the frequency with which $w_j^k$ (and the other words of length $|w_j^k|$ that occur when this word is self-concatenated) occur in $w_{2j-1}^{k+1}$ and $w_{2j}^{k+1}$ is at least $\delta_k$.

We claim that for each word $w_i^k$, the set of ergodic measures giving measure at least $\Delta$ to the set $[w_i^kw_i^k]$ is nonempty.  Moreover, we claim that the set of ergodic measures giving measure at least $\Delta$ to the set 
$$ 
\bigcup_{j=0}^{|w_i^k|}\sigma^j[\underbrace{w_i^kw_i^k\cdots w_i^k}_{S_k+1\text{ times}}] 
$$ 
is nonempty and that the set of ergodic measures giving measure at least $\Delta$ to the set 
$$ 
\bigcup_{j=0}^{|\overbrace{w_i^kw_i^k\cdots w_i^k}^{S_k\text{ times}}w_{i+1}^k|}\sigma^j[\underbrace{(\overbrace{w_i^kw_i^k\cdots w_i^k}^{S_k\text{ times}}w_{i+1}^k)(\overbrace{w_i^kw_i^k\cdots w_i^k}^{S_k\text{ times}}w_{i+1}^k)\cdots(\overbrace{w_i^kw_i^k\cdots w_i^k}^{S_k\text{ times}}w_{i+1}^k)}_{N_k\text{ times}}] 
$$ 
is also nonempty.  The first claim follows from either of the latter two.  We show the former, the latter being similar.  

Observe that $w_{2i}^{k+1}$ has a periodic region which is a long series of self-concatenations of
 $w_i^k$, that $w_{4i}^{k+2}$ is has a periodic region which is a long series of self-concatenations of $w_{2i}^{k+1}$, and similarly $w_{2^ji}^{k+j}$ has a periodic region which is a long series of self-concatenations of
  $w_{2^{j-1}i}^{k+j-1}$.  As these are all words in the language of $X$ and $X$ is closed, there is an element of $X$ for which the natural frequency of $w_i^k$ (and the other words of length $|w_i^k|$ that occur when it is self-concatenated) is at least $\prod_{j=k}^{\infty}\delta_j>\Delta$.  This follows because $w_i^k$ (and the other words of length 
  $|w_i^k|$ that occur when it is self-concatenated) occur with frequency at least $\delta_k$ in $w_{2i}^{k+1}$, 
  and inductively occurs with frequency at least $\prod_{j=k}^{\ell}\delta_j$ in $w_{w^{\ell-j}i}^{\ell}$ for any $\ell$. Therefore there is an invariant probability measure on $X$ that gives the union of these cylinder sets measure at least $\Delta$ and so there must be at least one ergodic measure that also has this property.  The claim follows. 

We next show that the set of ergodic measures on $X$ is uncountable (in fact with cardinality $\mathfrak{c}$).  We have shown that for each $k$ and $1\leq i\leq 2^k$,  there are two disjoint sets of ergodic measures giving large measure to the word $w_i^k$ and its periodic shifts.  The first (which we refer to as {\em type 0}) gives large measure to $\underbrace{w_i^kw_i^k\ldots w_i^k}_{N_k\text{ times}}$ (and its periodic shifts), 
whereas the second (which we 
refer to as {\em type 1}) gives small measure to this set and large measure to 
$$ 
\underbrace{(\overbrace{w_i^kw_i^k\ldots w_i^k}^{S_k\text{ times}}w_{i+1}^k)(\overbrace{w_i^kw_i^k\ldots w_i^k}^{S_k\text{ times}}w_{i+1}^k)\ldots(\overbrace{w_i^kw_i^k\ldots w_i^k}^{S_k\text{ times}}w_{i+1}^k)}_{N_k\text{ times}} 
$$ 
(and its periodic shifts).  Fix an infinite sequence $a_0, a_1,\dots$ of $0$'s and $1$'s.  For each $j$, let $\nu_j$ be an element of the set of ergodic measures on $X$ that start by giving large measure to $w_1^1$ and then are of  type $a_t$ for each $t=1,2,\dots,j$.  Let $\nu$ be a weak-* limit of a subsequence of these measures so that $\nu$ gives measure at least $\Delta$ to the word defining its type for each $t=1,2,\dots$  If we had chosen any other sequence of $0$'s and $1$'s it would have differed from $(a_t)$ at some finite stage and so there would be a cylinder set the measure resulting from that sequence gives large measure to which was given small measure by $\nu$.  Therefore each infinite sequence of $0$'s and $1$'s produces its own measure $\nu$.  Now, returning to the measure $\nu$, consider the sequence $(a_t)$ and the associate union of cylinder sets (which we call $\mathcal{S}_t$) that are given measure at least $\prod_{k=t}^{\infty}\delta_k$ by $\nu$.  Define 
$$ 
A_{\nu}=\bigcap_{t=1}^{\infty}\mathcal{S}_t. 
$$ 
Note that 
$$ 
\nu\left(\left(\bigcap_{t=1}^{k-1}\mathcal{S}_t\right)\setminus\left(\bigcap_{t=1}^k\mathcal{S}_t\right)\right)\leq1-\delta_k 
$$ 
and so we can arrange that 
$$ 
\nu\left(\bigcap_{t=1}^k\mathcal{S}_t\right)>8/10 
$$ 
for all $k$ by choosing $\delta_k$ to tend to $1$ sufficiently rapidly.  It follows that $\nu(A_{\nu})\geq8/10$.  Therefore there is an ergodic measure giving the set $A_{\nu}$ measure at least $8/10$ and so there is no loss of generality in assuming that $\nu$ is ergodic.  Therefore the set of ergodic measures is uncountable. 

This completes the proof of Theorem~\ref{thm:main}.  We note that by modifying the initial words $w_1^1$ and $w_2^1$, 
we can achieve the same result but ensure that the language has balanced numbers of short patterns.  More precisely, replacing  the initial use of $0$ by the word $01100110$ and the initial use of  $1$ by the word $11100100$ and carrying out the same construction, we have a system in which the average number of $0$'s and $1$'s on any short range is approximately one half.  This follows because all words later constructed are concatenations of $w_1^1$ and $w_2^1$ and so any word in the language of the shift of length larger than fourteen can be made into a word that is a concatenation of $w_1^1$'s and $w_2^1$'s by removing at most seven letters from each side of it (and this slightly shorter word has precisely the same number of $0$'s as $1$'s).

\section{Measure-theoretic complexity} 
\label{sec:measure-theoretic}
\subsection{Definition of slow entropy}
We review the definition of slow entropy, as defined by Katok and Thouvenot~\cite{KT} and Ferenczi~\cite{Fer}, 
adopting Ferenczi's notation in a way more convenient for our setting. 

Assume that $(X, \sigma)$ is a subshift.  
For $u,v\in\mathcal{L}_n(X)$, the {\em Hamming distance $d_H(u,v)$ between $u$ and $v$} is
$$ 
d_H(u,v)=\frac{|\{0\leq i<n\colon u_i\neq v_i\}|}{n}, 
$$ 
and this defines a metric on $\mathcal{L}_n(X)$.  For fixed $\varepsilon>0$ and $u\in\CL_N(X)$, define the {\em ball $B_\varepsilon(u)$ of radius $\varepsilon$ around $u$} by 
$$ 
B_{\varepsilon}(u)=\{v\in\mathcal{L}_n(X)\colon d_H(u,v)<\varepsilon\}. 
$$ 

Further assume that $\mu$ is an invariant measure on the shift $(X, \sigma)$.  When slow entropy is defined in~\cite{KT} and~\cite{Fer}, they consider an arbitrary measure preserving system and so a generating partition 
is a necessary ingredient.  As we are restricting ourselves to symbolic systems, we can assume
that the space $X$ is partitioned into cylinder sets of length one.  We implicitly make this assumption throughout and omit the partition from the notation.  

Define $K(n,\varepsilon,\sigma)$ to be the minimum number of words $u_1,u_2,\dots,u_k\in\mathcal{L}_n(X)$ such that 
$$ 
\mu\left(\bigcup_{i=1}^kB_{\varepsilon}(u_i)\right)>1-\varepsilon. 
$$ 
If $\varepsilon_1<\varepsilon_2$, then for any $u\in\mathcal{L}_n(X)$, we have
$B_{\varepsilon_1}(u)\subseteq B_{\varepsilon_2}(u)$.  
Thus $K(n,\varepsilon_2,\sigma)\leq K(n,\varepsilon_1,\sigma)$, 
meaning that $K(n,\varepsilon,\sigma)$ increases as $\varepsilon$ decreases.  If $(c_n)_{n\in\N}$ is a non-decreasing sequence of positive integers with $c_n\to\infty$, we say $P_{\sigma}^-(n)\succ(c_n)$ if 
$$ 
\lim_{\varepsilon\to0}\liminf_{n\to\infty}\frac{K(n,\varepsilon,\sigma)}{c_n}\geq1. 
$$ 
Similarly, we say that $P_{\sigma}^-(n)\prec(c_n)$ if 
$$ 
\lim_{\varepsilon\to0}\liminf_{n\to\infty}\frac{K(n,\varepsilon,\sigma)}{c_n}\leq1. 
$$ 
The analogous limits with $\liminf$ replaced by $\limsup$ define the conditions that $P_{\sigma}^+(n)\succ(c_n)$ and $P_{\sigma}^+(n)\prec(c_n)$, respectively.  It is shown in~\cite{KT} and~\cite{Fer} that for any fixed sequence $(c_n)$, the statement $P_{\sigma}^-(n)\prec(c_n)$ is a measure theoretic conjugacy invariant for $(X,\mu,\sigma)$ (as is 
the analogous statement for $P_{\sigma}^+(n)\succ(c_n)$).  
%
%

These notions of $P_{\sigma}^-(n)$ and $P_{\sigma}^+(n)$ clarify the statement of 
Theorem~\ref{thm:slow-entropy}, and we reproduce the statement: 
\begin{theorem*}[Theorem~\ref{thm:slow-entropy}]
Assume $(a_n)_{n\in\N}$ and $(b_n)_{n\in\N}$ are two nondecreasing sequences of positive integers such that $\lim a_n=\infty$, $\lim\frac{1}{n}\cdot\log(b_n)=0$, and $a_n\leq b_n$ for all $n\in\N$.  There exists a minimal subshift $(X_{\infty},\sigma)$ of topological entropy zero and an ergodic measure $\mu$ supported on $X_{\infty}$ such that 
$$ 
P_{\sigma}^-(n)\prec(a_n) \text{ and } P_{\sigma}^+(n)\succ(b_n). 
$$ 
\end{theorem*} 
Before turning to the proof, we 
make a few remarks to place the result in context.  
Ferenczi showed  the following: 
\begin{theorem}[{Ferenczi~\cite[Proposition 3]{Fer}}]\label{thm:ferenczi} 
Let $(X,\sigma)$ be a subshift and suppose $\mu$ is an ergodic measure supported on $X$.  Then the following are equivalent: 
	\begin{enumerate}
	\item $(X,\mu,\sigma)$ is a Kronecker system; \\ 
	\item $P_{\sigma}^-(n)\prec(c_n)$ for any nondecreasing sequences $(c_n)$ that tends to infinity; \\ 
	\item $P_{\sigma}^+(n)\prec(c_n)$ for any nondecreasing sequences $(c_n)$ that tends to infinity. 
	\end{enumerate} 
\end{theorem} 
This means that the assumption on the lower bound $(a_n)_{n\in\N}$ in Theorem~\ref{thm:slow-entropy}
can not be lowered as long as we still require that $P_{\sigma}^+(n)\succ(b_n)$, as this second condition implies that $(X,\mu,\sigma)$ is not a Kronecker system and so there must be some sequence $(c_n)_{n\in\N}$ that tends to infinity 
and is such that $P_{\sigma}^-(n)\succ (c_n)$. 

At the other extreme, Katok showed: 
\begin{theorem}[{Katok\cite[Theorem 1.1]{Kat}}]\label{thm:katok}  
Let $(X,\sigma)$ be a subshift and suppose $\mu$ is an ergodic measure supported on $X$.  Then the following are equivalent: 
	\begin{enumerate}
	\item $(X,\mu,\sigma)$ has positive entropy; \\ 
	\item there exists $\lambda>1$ such that $P_{\sigma}^-(n)\succ(\lambda^n)$; \\ 
	\item there exists $\lambda>1$ such that $P_{\sigma}^+(n)\succ(\lambda^n)$. 
	\end{enumerate} 
\end{theorem} 
In particular, since $(a_n)_{n\in\N}$ grows subexponentially and $P_{\sigma}^-(n)\prec(a_n)$, this implies that $(X,\mu,\sigma)$ has zero entropy and so we cannot have $P_{\sigma}^+(n)\succ(b_n)$ for any sequence $(b_n)_{n\in\N}$ with positive exponential growth rate.  Theorem~\ref{thm:slow-entropy} implies that, even subject to the requirement that $P_{\sigma}^-(n)\prec(a_n)$, $P_\sigma^+(n)$ can grow as quickly as we want, subject to the necessary condition that it grow subexponentially, as given by Katok's Theorem.

Katok and Thouvenot~\cite{KT} and Serafin~\cite{S} give constructions of a subshift with zero topological entropy
and $P_{\sigma}^+(n)$ growing arbitrarily quickly (of course still subject to the condition that the growth be subexponential), but without a requirement that $P_{\sigma}^-(n)$ grow slowly.  
In showing that we achieve our upper bound (what we refer to as the loud phase), our 
construction and the derivation of its properties has many features in common with their constructions.  However, as we generalize this approach, we include all details for the sake of clarity.  

\subsection{Large-scale features of the construction} 
We fix the  nondecreasing sequences $(a_n)_{n\in\N}$ and $(b_n)_{n\in\N}$ 
of positive integers such that $\lim_{n\to\infty} a_n=\infty$, $\lim_{n\to\infty}\frac{1}{n}\cdot\log(b_n)=0$, and $a_n\leq b_n$ for all $n\in\N$.
Let $(\varepsilon_n)_{n\to\infty}$ be a decreasing sequence of positive real numbers such that $\lim_{n\to\infty}\varepsilon_n=0$.  We inductively construct a descending sequence of positive entropy subshifts: 
$$ 
\A^{\Z}=:X_0\supseteq X_1\supseteq X_2\supseteq X_3\supseteq\cdots\supseteq X_n\supseteq\cdots 
$$ 
and an increasing sequence of positive integers 
$$ 
N_1<P_1<N_2<P_2<N_3<P_3<\cdots.
$$ 
In our construction, we show that 
$$ 
X_{\infty}:=\bigcap_{i=1}^{\infty}X_i 
$$ 
is nonempty and show that if $\mu$ is any ergodic measure supported on $X_{\infty}$, then 
\begin{equation}\label{eq1} 
K(N_i,1/8,\sigma)>b_{N_i} 
\end{equation} 
and 
\begin{equation}\label{eq2} 
K(P_i,\varepsilon_i,\sigma)\leq a_{P_i}. 
\end{equation} 
Since $K(n,\varepsilon,\sigma)$ increases as $\varepsilon$ decreases, 
it follows that for sufficient small $\varepsilon> 0$, 
$$ 
\limsup_{n\to\infty}\frac{K(n,\varepsilon,\sigma)}{b_n}\geq\limsup_{i\to\infty}\frac{K(N_i,\varepsilon,\sigma)}{b_{N_i}}\geq\limsup_{i\to\infty}\frac{K(N_i,1/8,\sigma)}{b_{N_i}}
\geq1, 
$$ 
meaning that 
$$ 
\lim_{\varepsilon\to0}\limsup_{n\to\infty}\frac{K(n,\varepsilon,\sigma)}{b_n}\geq1, 
$$ 
and so $P_{\sigma}^+(n)\succ(b_n)$. 
Similarly, since $\lim_{i\to\infty}\varepsilon_i=0$ and again passing to the subsequence $(P_i)_{i\in\N}$, it follows that 
$$ 
\liminf_{n\to\infty}\frac{K(n,\varepsilon,\sigma)}{a_n}\leq\liminf_{i\to\infty}\frac{K(P_i,\varepsilon,\sigma)}{a_{P_i}}\leq\liminf_{i\to\infty}\frac{K(P_i,\varepsilon_i,\sigma)}{a_{P_i}}\leq1. 
$$ 
This means that 
$$ 
\lim_{\varepsilon\to0}\liminf_{n\to\infty}\frac{K(n,\varepsilon,\sigma)}{a_n}\leq1 
$$ 
and so $P_{\sigma}^-(n)\prec(a_n)$.  Therefore,  to prove Theorem~\ref{thm:slow-entropy}, 
it suffices to construct the shift $X_{\infty}$, the ergodic measure $\mu$ and show that they satisfy~\eqref{eq1} and~\eqref{eq2}.  To show that this construction can be carried out to produce a minimal shift, we note that by the Jewett-Krieger Theorem there is a strictly ergodic model for $(X_{\infty},\mu,\sigma)$ and this model must obey the same slow entropy bounds because they are invariants of measure theoretic conjugacy.  We start in Section~\ref{subsec:estimates} by providing estimates on how the words in the language of the shift must be constructed and then in Section~\ref{subsec:proof-of-main} complete the construction of the subshift and verify its properties.

\subsection{Estimates for the language} 
\label{subsec:estimates}

We start with a lemma for use in the proof of the main theorem: 
\begin{lemma}\label{lem:combinatorics} 
Let $k,n\in\N$ be fixed and let $A_1,\dots,A_{2k-1}\subseteq\{1,2,\dots,n\}$ be any collection of subsets satisfying $|A_i|\geq n/2$ for $i=1, \ldots, 2k-1$.  Then there exist distinct indices $1\leq i_1<i_2<\cdots<i_k\leq2k-1$ and $s\in\{1,2,\dots,n\}$ such that $s\in A_{i_j}$ for all $j=1,2,\dots,k$. 
\end{lemma} 
\begin{proof} 
For contradiction, suppose $A_1,\dots,A_{2k-1}\subseteq\{1,2,\dots,n\}$ are a collection of subsets satisfying $|A_i|\geq n/2$ for $i=1, \ldots, 2k-1$, but no subcollection of $k$ of these sets have nonempty intersection.  For each $x\in\{1,2,\dots,n\}$, let $i(x)$ denote the number of distinct indices $j$ such that $x\in A_j$.  Then $i(x)\leq k-1$ for all $x$ and so 
	\begin{align*} 
	n(k-1)&<\sum_{i=1}^{2k-1}|A_i|  =\sum_{m=0}^{k-1}m\cdot|\{x\colon i(x)=m\}| \\ 
&\leq \sum_{m=0}^{k-1}(k-1)\cdot|\{x\colon i(x)=m\}| 
	=n(k-1), 
	\end{align*} 
where the last equality holds since the level sets of $i(x)$ partition $\{1,2,\dots,n\}$.  Thus no such collection of sets exists. 
\end{proof} 

\begin{lemma}\label{lem:quiet} 
Let $\mathcal{A}$ be a finite alphabet and suppose $k,N,M\in\N$ are fixed and assume that $M>1$.  Let $w_1,\dots,w_k\in\mathcal{L}_N(\A^{\Z})$ and for each $1\leq i\leq k$, let 
$$ 
v_i=\underbrace{w_iw_iw_i\cdots w_i}_{\text{$M$ times}}. 
$$ 
Let $(X,\sigma)$ denote the subshift (with alphabet $\CA$) defined by the words $v_1,\dots,v_k$.   Assume $P\in[1,NM)$ is  an integer and let $\mathcal{W}_P\subseteq\mathcal{L}_P(X)$ be the set of all words $u$ for which there exists $i$ such that $u$ is a subword of $v_i$.  Then if $\mu$ is any ergodic measure supported on $X$, we have 
$$ 
\mu\left(\bigcup_{w\in\mathcal{W}_P}[w]\right)\geq1-\frac{P-1}{NM}.  
$$ 
\end{lemma} 
\begin{proof} 
Assume $\mu$ is an ergodic measure supported on $X$ 
and set $\mathcal{S}=\bigcup_{w\in\mathcal{W}_P}[w]$.  By the pointwise ergodic theorem, there exists $x\in X$ such that 
$$ 
\mu(\mathcal{S})=\lim_{n\to\infty}\frac{1}{2n+1}\sum_{i=-n}^n\one_{\mathcal{S}}(\sigma^ix). 
$$ 
By definition of the subshift $(X, \sigma)$, the element $x$ can be parsed into a bi-infinite concatenation of the words $v_1,\dots,v_k$.  Fix one such way to parse $x$ and 
let $\mathcal{I}\subseteq\Z$ be the set of indices at which these words begin; note that $\mathcal{I}$ is an arithmetic progression in $\Z$ with gap $NM$.   
For each $i\in\Z$, the cylinder set of length $P$ that contains $\sigma^ix$ is contained in $\mathcal{S}$, unless $i$ lies within distance $P-1$ of the smallest element of $\mathcal{I}$ larger than $i$.  
Thus for any $n\in\N$, 
$$ 
\sum_{i=-n}^n\one_{\mathcal{S}}(\sigma^ix)\geq2n+1-(P-1)\cdot|\mathcal{I}\cap[-n,n]|. 
$$ 
Since $\mathcal{I}$ is an arithmetic progression with gap $NM$, it follows that 
\begin{equation*}
\mu(\mathcal{S})\geq\lim_{n\to\infty}\frac{2n+1-(P-1)\cdot|\mathcal{I}\cap[-n,n]|}{2n+1}\geq1-\frac{P-1}{NM}.\ \qedhere
\end{equation*}
\end{proof} 

\begin{lemma}\label{lem:close-measure}  
Let $N\in\N$ and $\varepsilon>0$ be fixed.  For each $n$, let $\mathcal{W}_n\subseteq\{1,2,\dots,N\}^n$ be the set of words $w=(w_0,w_1,\dots,w_{n-1})$ for which 
$$ 
(1-\varepsilon)\frac{n}{N}<|\{i\in[0,n)\colon w_i=a\}|<(1+\varepsilon)\frac{n}{N} 
$$ 
for all $a\in\{1,2,\dots,N\}$.  Then there exists $M$ such that for all $n>M$, we have $|\mathcal{W}_n|>(1-\varepsilon)N^n$. 
\end{lemma} 
\begin{proof} 
Let $\nu$ be the $(1/N,1/N,\dots,1/N)$-Bernoulli measure on $\{1,2,\dots,N\}^{\N}$.  By the pointwise ergodic theorem, for almost all $x\in\{1,2,\dots,N\}^{\N}$ and for each $a\in\{1,2,\dots,N\}$ we have 
$$ 
\lim_{n\to\infty}\frac{1}{n}\sum_{i=0}^{n-1}\one_{[a]}(\sigma^ix)=\nu([a])=\frac{1}{N}. 
$$ 
Therefore there exists $M(x)$ such that for all $n>M(x)$ we have 
$$ 
\frac{1-\varepsilon}{N}<\frac{1}{n}\sum_{i=0}^{n-1}\one_{[a]}(\sigma^ix)<\frac{1+\varepsilon}{N} 
$$ 
for all $a\in\{1,2,\dots,N\}$.  Thus there is some $M$ and a set $\mathcal{S}$ of $\nu$-measure at least $1-\varepsilon$ such that these inequalities hold for any $x\in\mathcal{S}$ and any $n>M$.  
Setting $w_n(x)=(x_0,x_1,\dots,x_{n-1})$, we have that 
$$ 
\nu\left(\bigcup_{x\in\mathcal{S}}[w_n(x)]\right)\geq\nu(\mathcal{S})\geq1-\varepsilon. 
$$ 
Since the $\nu$-measure of each word of length $n$ is $1/N^n$, 
it follows that the number of distinct words of length $n$ that can be written as $w_n(x)$ for some $x\in\mathcal{S}$ is at least $(1-\varepsilon)N^n$. 
\end{proof} 

We combine these to derive our key estimate on the statistics in the language: 
\begin{proposition}\label{prop:stats} 
Let $N\in\N$, $\varepsilon>0$, and $0<\alpha<\frac{N-1}{N}$ be fixed.  Then there exists $M\in\N$ and $\lambda=\lambda(N,\varepsilon,\alpha)>1$ such that for any $n>M$,  there is a set of words $w_1,\dots,w_k\in\{1,2,\dots,N\}^n$ satisfying 
$$ 
d_H(w_i,w_j)>\alpha \text{ for all } i\neq j 
$$ 
with $k>\lambda^n$ and for all $a\in\{1,2,\dots,N\}$, 
$$ 
(1-\varepsilon)\frac{n}{N}<|\{i\in[0,n)\colon w_i=a\}|<(1+\varepsilon)\frac{n}{N}. 
$$ 
Moreover, these words can be chosen such that for any $1\leq j_1<j_2\leq k$, no word of length $n$ that occurs as a subword of $w_{j_1}w_{j_1}$ is also a subword of $w_{j_2}w_{j_2}$. 
\end{proposition} 
\begin{proof} 
Let $w\in\{1,2,\dots,N\}^n$ be fixed.  A classical use of Stirling's Formula (see for example~\cite[Equation (1.3)]{Kat}) shows that, since $0<\alpha<\frac{N-1}{N}$, 
\begin{eqnarray*} 
\lim_{n\to\infty}\frac{1}{n}\log|\{u\in\{1,2,\dots,N\}^n\colon d_H(u,w)<\alpha\}| \\ 
&&\hspace{-1 in}=\alpha\log(N-1)-\alpha\log\alpha-(1-\alpha)\log(1-\alpha). 
\end{eqnarray*} 
Set  $f(x)=x\log(N-1)-x\log x-(1-x)\log(1-x)$.  Then 
for $x<(N-1)/N$, the derivative of $f(x)$ is positive and $\lim_{x\to[(N-1)/N]^-}f(x)=\log(N)$.  
Thus $f(\alpha)<\log(N)$, and so there exists $M\in\N$ and $\delta>0$ 
sufficiently small such that $(1+\delta)f(\alpha)<\log(N)$ and such that for all $n>M$, 
$$ 
|\{u\in\{1,2,\dots,N\}^n\colon d_H(u,w)<\alpha\}|<2^{n\cdot(1+\delta)f(\alpha)}.  
$$ 
By Lemma~\ref{lem:close-measure}, if $\mathcal{W}_n$ is the set of all words $w=(w_0,w_1,\dots,w_{n-1})\in\{1,2,\dots,N\}^n$ for which 
$$ 
(1-\varepsilon)\frac{n}{N}<|\{i\in[0,n)\colon w_i=a\}|<(1+\varepsilon)\frac{n}{N}, 
$$ 
then $|\mathcal{W}_n|>(1-\varepsilon)N^n$ for all sufficiently large $n$.  Adjusting the value of $M$ if necessary, we can assume this holds for all $n>M$.  But for each $w\in\mathcal{W}_n$, we have 
$$ 
|\{u\in\mathcal{W}_n\colon d_H(u,w)<\alpha\}|\leq|\{u\in\{1,2,\dots,N\}^n\colon d_H(u,w)<\alpha\}|\leq2^{n\cdot(1+\delta)f(\alpha)}. 
$$ 
This means there is a set of at least 
$$ 
\left\lfloor\frac{(1-\varepsilon)N^n}{2^{n\cdot(1+\delta)f(\alpha)}}\right\rfloor 
$$ 
elements of $\mathcal{W}_n$ that are pairwise at least $\alpha$ separated in the  Hamming distance.  
If $u,v$ are two words in this set and if some word $w$ of length $|u|=|v|$ occurs as a subword of $uu$ and $vv$, then $u$ is itself a subword of $vv$.  Thus there is a subset of size at least 
$$ 
k(n):=\frac{1}{n}\cdot\left\lfloor\frac{(1-\varepsilon)N^n}{2^{n\cdot(1+\delta)f(\alpha)}}\right\rfloor 
$$ 
with the additional property that for any $u,v$ in this list, no word of length $n$ occurs as a subword of both $uu$ and $vv$.  

Since $(1+\delta)f(\alpha)<\log(N)$, it follows that 
$$ 
g:=\lim_{n\to\infty}\frac{\log k(n)}{n}=\log(N)-(1+\delta)f(\alpha)>0.  
$$ 
Thus if $\lambda:=2^{g/2}$, then $\lambda>1$ and $k(n)\geq\lambda^n$ for all $n>M$. 
\end{proof} 

\subsection{Construction of the subshifts} 
\label{subsec:proof-of-main}
The remainder of this section is devoted to proof of Theorem~\ref{thm:slow-entropy}.  

We construct a sequence of subshifts inductively.  Each step of the construction 
involves the construction of two nested subshifts, with the larger one being referred to as the ``noisy'' phase and the smaller one being referred to as the ``quiet'' phase.  At each stage of the construction, 
we appeal to Proposition~\ref{prop:stats}, 
and this necessitates the definition of  two auxiliary sequences.  
Set $\alpha_0=1/3$ and let $(\alpha_i)_{i>0}$ be an increasing sequence of real numbers with $0<\alpha_i<1$ for all $i\in\N$ and such that 
\begin{equation}\label{eq:alpha} 
\prod_{i=1}^{\infty}\alpha_i>\frac{3}{4}. 
\end{equation} 
This sequence provides the parameter $\alpha$ appearing in Proposition~\ref{prop:stats} for each step of the construction.  Next let $\{\varepsilon_i\}_{i\geq0}$ be a decreasing sequence of real numbers with $0<\varepsilon_i<1$ for all $i$ and such that 
\begin{equation}\label{eq:def-of-ep} 
\prod_{i=1}^{\infty}(1-\varepsilon_i)>\frac{99}{100}. 
\end{equation} 
This sequence provides the parameter $\varepsilon$ appearing Proposition~\ref{prop:stats}. 

\subsubsection*{Base Loud Phase}
Let $X_0:=\{1,2\}^{\Z}$ and $M_0=1$.  Thus $X_0$ is a subshift on $N_0:=2$ letters.  
Since $1/3=\alpha_0<\frac{N_0-1}{N_0}$, we can apply Proposition~\ref{prop:stats} with parameters $N=N_0$, $\varepsilon=\varepsilon_0$, and $\alpha=\alpha_0$.  
Thus there exists $\lambda_0>1$ such that for any sufficiently large integer $n$, 
we can find a set of words $w_1(n),w_2(n),\dots,w_{k(n)}(n)\in\{1,2\}^n$ where $k(n)>\lambda_0^n$ and such that for each $a\in\{1,2\}$, the following conditions are satisfied: 
\begin{enumerate}
\item We have the estimate $ (1-\varepsilon_0)\frac{n}{2}<B_a<(1+\varepsilon_0)\frac{n}{2}$ for all $i$, where $B_a$ denotes the number of locations where the letter $a$ occurs in $w_i(n)$;
\item We have the distances separated, meaning that $ d_H(w_i(n),w_j(n))>\alpha_0$ for any $i\neq j$;
\item No word of length $n$ occurs as a subword of both $w_i(n)w_i(n)$ and $w_j(n)w_j(n)$ for some $i\neq j$.  
\end{enumerate}
Let $N_1$ be an integer which is sufficiently large 
that we can choose such a set of words, such that $\lambda_0^{N_1}>4b_{N_1}$, and such that $\alpha_1<\frac{N_1-1}{N_1}$.  Let $k_1:=k(N_1)$ be the number of words constructed in this way and let $w_1,\dots,w_{k_1}\in\{1,2\}^{N_1}$ be the words produced by the construction.  Finally let $L_1\subseteq X_0$ be the subshift of $X_0$ defined by the words  $w_1,w_2,\dots,w_{k_1}$.

\subsubsection*{Base Quiet Phase} Choose an integer $P_1>N_1$ sufficiently large that 
$$ 
|w_1|+|w_2|+\cdots+|w_{k_1}|=k_1N_1<a_{P_1} 
$$ 
and then choose an integer $M_1$ such that $(P_1-1)/(N_1M_1)<\varepsilon_0$.  
For each $1\leq i\leq k_1$, define 
$$ 
v_i:=\underbrace{w_iw_iw_i\cdots w_i}_{M_1\text{ times}}. 
$$ 
Let $X_1\subseteq L_1$ be the subshift defined by the words $v_1,\dots,v_{k_1}$.  
We apply Lemma~\ref{lem:quiet} with parameters $\mathcal{A}=\{1,2\}$, $k=k_1$, $N=N_1$, $M=M_1$, and  
$P=P_1\in[1, NM)$.  If $\CW = \mathcal{W}_{P_1}$ is the set of words of length $P_1$ that occur as subwords of the words $v_1,\dots,v_{k_1}$, then for any ergodic measure $\mu$ supported on $X_1$ we have 
$$ 
\mu\left(\bigcup_{w\in\mathcal{W}}[w]\right)>1-\frac{P_1}{N_1M_1}>1-\varepsilon_0 
$$ 
and if $\widetilde{\mathcal{W}}= \widetilde{\mathcal{W}}_{N_1}$ is the set of words of length $N_1$ that occur as subwords of the words $v_1,\dots,v_{k_1}$, then 
$$ 
\mu\left(\bigcup_{w\in\widetilde{\mathcal{W}}}[w]\right)>1-\frac{N_1}{N_1M_1}>1-\varepsilon_0. 
$$

\subsubsection*{Inductive Loud Phase} Assume we have constructed subshifts 
$$ 
\{1,2\}^{\Z}=:X_0\supseteq L_1\supseteq X_1\supseteq L_2\supseteq X_2\supseteq\cdots\supseteq L_i\supseteq X_i , 
$$ 
a sequence of integers $N_1<P_1<N_2<P_2<\cdots<N_i<P_i$, a sequence of integers $M_1<M_2<\cdots<M_i$, such that that following hold: for each $1\leq j\leq i$, 
	\begin{enumerate}
	\item $N_j$ is sufficiently large such that $\alpha_j<\frac{N_j-1}{N_j}$; \label{c1}
	\item There exists $\lambda_j>1$ and an integer $k_j$ such that $k_j>\lambda_j^{N_j}>4b_{N_j}$; \label{c2} 
	\item There exist words $w_1^j,w_2^j,\dots,w_{k_j}^j\in\{1,2\}^{N_j\cdot\prod_{s=0}^{j-1}N_sM_s}$ such that $L_j$ 
	is comprised of all elements of $X_{j-1}$ defined by the words $w_1^j,\dots,w_{k_j}^j$ and for $i_1\neq i_2$ we have $d_H(w_{i_1}^j,w_{i_2}^j)>\prod_{s=1}^j\alpha_s$ and additionally no word of length $|w_1^j|$ occurs as a subword of both $w_{i_1}^jw_{i_1}^j$ and $w_{i_2}^jw_{i_2}^j$; \label{c3} 
	\item For each $1\leq t\leq k_j$, there is a word 
	$$ 
	v_t^j:=\underbrace{w_t^jw_t^jw_t^j\cdots w_t^j}_{M_j\text{ times}} 
	$$ 
	where for any $t_1\neq t_2$ no subword of length $N_j\cdot\prod_{s=0}^{j-1}N_sM_s$ in $v_{t_1}^j$ is also a subword of $v_{t_2}^j$, and $X_j\subseteq L_j$ is the subshift of $L_j$ defined by the 
	words $v_1^j,v_2^j,\dots,v_{k_j}^j$; \label{c4} 
	\item \label{c5} 
	If $j>1$, then for each $1\leq t\leq k_j$ the word $w_t^j$ can be written as a concatenation of the words $v_1^{j-1},\dots,v_{k_{j-1}}^{j-1}$ and so by identifying this set of $k_{j-1}$ words with the alphabet $\{1,2,\dots,k_{j-1}\}$, 
	we can identify 
	$w_t^j$ with a word of length $N_j$ written in these letters.  With this identification, for each $a\in\{1,2,\dots,k_{j-1}\}$, 
	 we have 
	\begin{align*}
	\frac{N_j}{k_{j-1}}\cdot(1-\varepsilon_{j-1}) & <\bigl|\{i\in[0,n)\colon\text{ the $i^{th}$ letter in }w_t^j\text{ is }a\}\bigr|\\ & <\frac{N_j}{k_{j-1}}\cdot(1+\varepsilon_{j-1}); 
\end{align*}
	\item We have $P_j>k_j$, $M_j>|w_1^i|k_j/\varepsilon_j$, and $(P_j-1)/(N_jM_j)<\varepsilon_j$ and if $\mathcal{W} = \CW_{P_j}$ is the set of all words of length $P_j$ that occur in $X_j$ as subwords of $v_1^j,v_2^j,\dots,v_{k_j}^j$ and if $\mu$ is any ergodic measure supported on $X_j$, then 
	$$ 
	\mu\left(\bigcup_{w\in\mathcal{W}}[w]\right)>(1-\varepsilon_j) 
	$$ 
	and if $\widetilde{\CW} = \widetilde{\CW}_{N_j}$ is the set of all words of length $N_j$ that occur in $X_j$ as subwords of $v_1^j,v_2^j,\dots,v_{k_j}^j$ then 
	$$ 
	\mu\left(\bigcup_{w\in\widetilde{\mathcal{W}}}[w]\right)>(1-\varepsilon_j). 
	$$ \label{c6} 
	\end{enumerate} 
Since $\alpha_i<\frac{N_i-1}{N_i}$, we can apply Proposition~\ref{prop:stats} with parameters $N=k_i$, $\varepsilon=\varepsilon_i$, and $\alpha=\alpha_i$.  Thus 
there exists $\lambda_i>1$ such that for any sufficiently large integer $n$, 
there is a set of words $w_1(n),w_2(n),\dots,w_{k(n)}(n)\in\{1,2,\dots,k_i\}^n$ where $k(n)>\lambda_i^n$ and such that for any $a\in\{1,2,\dots,k_i\}$, the following conditions are satisfied: 
\begin{enumerate}
\item We have the estimate $ (1-\varepsilon_i)\frac{n}{k_i}<B_a<(1+\varepsilon_i)\frac{n}{k_i}$, 
where again $B_a$ denotes 
the number of locations where the letter $a$ occurs in $w_i(n)$;
\item 
We have the distances separated, meaning that $d_H(w_i(n),w_j(n))>\alpha_i$ for any $i\neq j$;
\item No word of length $n$ occurs as a subword of both $w_i(n)w_i(n)$ and $w_j(n)w_j(n)$ for any $i\neq j$. 
\end{enumerate}
Let $N_{i+1}$ be a sufficiently large integer such 
that we can find such a set of words, such that $\lambda_i^{N_{i+1}}>4b_{N_{i+1}}$, and such that $\alpha_{i+1}<\frac{N_{i+1}-1}{N_{i+1}}$.  Let $k_{i+1}:=k(N_{i+1})$ be the number of words constructed in this way.  Finally let 
$$ 
w_1^{i+1},w_2^{i+1},\dots,w_{k_{i+1}}^{i+1}\in\{1,2\}^{N_{i+1}\cdot\prod_{s=0}^iN_sM_s} 
$$ 
be the words constructed by concatenating $w_1^i,\dots,w_{k_i}^i$ according to the letters of the words $w_1(n),\dots,w_{k_{i+1}}(n)$: for $1\leq j\leq k_{i+1}$ if $w_j(n)=a_1a_2\cdots a_{k_i}$ then we define 
$$ 
w_j^{i+1}=w_{a_1}^iw_{a_2}^i\cdots w_{a_{k_i}}^i. 
$$ 
Finally let $L_{i+1}\subseteq X_i$ be the subshift of $X_i$ defined by  
the words $w_1^{i+1},w_2^{i+1},\dots,w_{k_{i+1}}^{i+1}$.

\subsubsection*{Inductive Quiet Phase}
Choose an integer $P_{i+1}>N_{i+1}$ sufficiently large such that 
\begin{equation}\label{eq:def-of-P} 
|w_1^{i+1}|+|w_2^{i+1}|+\cdots+|w_{k_{i+1}}^{i+1}|=k_{i+1}N_{i+1}\cdot\prod_{s=0}^i N_sM_s<a_{P_{i+1}}. 
\end{equation} 
Find an integer $M_{i+1}>|w_1^{i+1}|k_{i+1}/\varepsilon_{i+1}$ such that $(P_{i+1}-1)/(N_{i+1}M_{i+1})<\varepsilon_i$.  For each $1\leq j\leq k_{i+1}$, define 
$$ 
v_j^{i+1}:=\underbrace{w_j^{i+1}w_j^{i+1}w_j^{i+1}\cdots w_j^{i+1}}_{M_{i+1}\text{ times}}. 
$$ 
Let $X_{i+1}\subseteq L_{i+1}$ be the subshift of $L_{i+1}$ 
defined by the words $v_1^{i+1},\dots,v_{k_{i+1}}^{i+1}$.  
As in the base case, for each $1\leq j\leq k_{i+1}$, we choose $w_j^{i+1}=w_{a_1}^iw_{a_2}^i\cdots w_{a_{N_i}}^i$ as 
a way to parse $w_j^{i+1}$ into a concatenation of words with superscript $i$.  Then define 
$$ 
\tilde{v}_j^{i+1}=\underbrace{(a_1a_2\cdots a_{N_i})(a_1a_2\cdots a_{N_i})\cdots(a_1a_2\cdots a_{N_i})}_{M_{i+1}\text{ times}}
$$ 
to be the identification of $v_j^{i+1}$ with a concatenation of letters $\{1,2,\dots,k_i\}$, rather than words $\{w_1^i,w_2^i,\dots,w_{k_i}^i\}$.  Let $\widetilde{X}_{i+1}$ be the subshift of $\{1,2,\dots,k_i\}^{\Z}$ 
defined by the words $\tilde{v}_1^{i+1},\dots,\tilde{v}_{k_{i+1}}^{i+1}$.  
We then apply Lemma~\ref{lem:quiet} with parameters $\mathcal{A}=\{1,2,\dots,k_i\}$, $k=k_{i+1}$, $N=N_{i+1}$, $M=M_{i+1}$, and we choose $P=P_{i+1}\in [1, N_{i+1}M_{i+1})$.  Then 
if $\mathcal{W} = \CW_{P_{i+1}}$ denotes the set of words of length $P_{i+1}$ that occur as subwords of $\tilde{v}_1^{i+1},\tilde{v}_2^{i+1},\dots,\tilde{v}_{k_{i+1}}^{i+1}$ and if $\nu$ is any ergodic measure supported on $\tilde{X}_{i+1}$, 
then 
$$ 
\nu\left(\bigcup_{w\in\mathcal{W}}[w]\right)>1-\frac{P_{i+1}}{N_{i+1}M_{i+1}}>(1-\varepsilon_i). 
$$ 
If $\widetilde{\CW}= \widetilde{\CW}_{N_i}$ is the set of words of length $N_i$ that occur as subwords of $\tilde{v}_1^{i+1},\dots,v_{k_{i+1}}^{i+1}$, 
then 
$$ 
\nu\left(\bigcup_{w\in\widetilde{\mathcal{W}}}[w]\right)>1-\frac{N_{i+1}}{N_{i+1}M_{i+1}}>(1-\varepsilon_i). 
$$ 
Therefore conditions~\eqref{c1}--\eqref{c6} of the induction hypothesis are satisfied for $j=i+1$.

Thus, by induction, we obtain an infinite descending sequence of subshifts 
$$ 
\{1,2\}^{\Z}=:X_0\supseteq L_1\supseteq X_1\supseteq L_2\supseteq X_2\supseteq\cdots\supseteq L_i\supseteq X_i\supseteq\cdots 
$$ 
We define
$$ 
X_{\infty}:=\bigcap_{i=0}^{\infty}X_i. 
$$ 
Since  $\{1,2\}^{\Z}$ is a Baire space (with the usual metric), the intersection of any nested sequence of subshifts is nonempty and so  $X_{\infty}$ is nonempty.  

We now assume that $\mu$ is an ergodic measure supported on $X_{\infty}$ and we study its properties.  

\subsubsection*{Analysis of $P^+(n)$}  Our goal is to show that 
$$ 
K(\alpha,N_i,1/8,\sigma)>b_{N_i} 
$$ 
for all $i$.  Fix $i\in\N$.  First we recall the definition of $K(N_i,1/8,\sigma)$.  For $u\in\mathcal{L}_{N_i}(X_{\infty})$, 
let 
$$ 
B_{1/8}(u)=\{w\in\mathcal{L}_{N_i}(X_{\infty})\colon d_H(u,w)<1/8\} 
$$
be the $(1/8)$-Hamming ball around $u$.  Define  
$$ 
[B_{1/8}(u)]:=\bigcup_{w\in B_{1/8}(u)}[w]. 
$$ 
With this notation, $K(N_i,1/8,\sigma)$ is the smallest cardinality of a set $\mathcal{U}\subseteq\mathcal{L}_{N_i}(X_{\infty})$ such that 
\begin{equation}\label{eq:cyls} 
\mu\left(\bigcup_{u\in\mathcal{U}}[B_{1/8}(u)]\right)>\frac{7}{8}. 
\end{equation} 
Fix such a set $\mathcal{U}\subseteq\mathcal{L}_{N_i}(X_{\infty})$.  
To establish~\eqref{eq1}, we are left with showing that $|\mathcal{U}|>b_{N_i}$.  

Since $\mu$ is an ergodic measure supported on $X_{\infty}$, 
it is also an ergodic measure supported on $X_i$ (albeit not a measure of full support).  The shift 
$X_i$ is constructed by first constructing a set of words $w_1^i,w_2^i,\dots,w_{k_i}^i$ of length $N_i$ and using them to construct words $v_1^i,v_2^i,\dots,v_{k_i}^i$ via the formula 
$$ 
v_j^i=\underbrace{w_j^iw_j^iw_j^i\cdots w_j^i}_{M_i \text{ times}}
$$ 
and $M_i$ is a parameter chosen during the construction.  
Thus the language of $X_i$ is defined to be all elements of $\{1,2\}^{\Z}$ that can be written as bi-infinite concatenations of $v_1^i,v_2^i,\dots,v_{k_i}^i$.  The choice of $M_i$ guarantees, by induction hypothesis~\eqref{c6}, that for any ergodic measure supported on $X_i$ (in particular, for $\mu$) if $\widetilde{\mathcal{W}}$ is the set of words of length $N_i$ that occur as subwords of one of $v_1^i,v_2^i,\dots,v_{k_i}^i$, then 
\begin{equation}\label{eq:good-stats} 
\mu\left(\bigcup_{w\in\widetilde{\mathcal{W}}}[w]\right)>1-\varepsilon_i. 
\end{equation}  
The construction also guarantees, by induction hypothesis~\eqref{c3}, that for $j_1\neq j_2$ we have $d_H(w_{j_1}^i,w_{j_2}^i)>\prod_{s=0}^i\alpha_s\geq\prod_{s=0}^{\infty}\alpha_s>1/4$ (recall that $\alpha_0=1/3$ and equation~\eqref{eq:alpha}).  Now observe that from~\eqref{eq:cyls} and~\eqref{eq:good-stats}, if $\mathcal{V}\subseteq\widetilde{W}$ is the set of all $w\in\widetilde{\mathcal{W}}$ such that there exists $u\in\mathcal{U}$ such that $d_H(u,w)<1/8$, then 
$$ 
\mu\left(\bigcup_{w\in\mathcal{V}}[w]\right)>\frac{7}{8}-\varepsilon_i\geq\frac{3}{4} 
$$ 
provided $i$ is sufficiently large.  Next observe that if $w\in\mathcal{V}$, then $w$ is a word of length $N_i$ that occurs as a subword of one of $v_1^i,v_2^i,\dots,v_{k_i}^i$.  In particular, this means there exists $1\leq j\leq k_i$ such that $w$ is a subword of $w_j^iw_j^i$ (recall that $v_j^i$ is just the concatenation of a large number of copies of $w_j^i$ and $|w_j^i|=N_i$).  We write 
$$ 
[[w_j^i]]=\bigcup_{w\hookrightarrow w_j^iw_j^i}[w] 
$$ 
where $w\hookrightarrow w_j^iw_j^i$ means $w$ is a word of length $N_i$ that occurs as a subword of $w_j^iw_j^i$.  Therefore there exists a smallest integer $1\leq t(w)<N_i$ such that the subword of $w_j^iw_j^i$ of length $N_i$ that starts on the $t(w)^{th}$ letter is $w$; let $s(w)$ denote 
the set whose only element is $t(w)$ (if $w_j^iw_j^i$ is periodic of period smaller than $|w_j^i|$, let $s(w)$ denote the set of starting points of $w$, excluding $|w_j^i|$ if $w=w_j^i$).  Next recall that $X_{\infty}\subseteq L_{i+1}\subseteq X_i$.  The subshift $L_{i+1}$ is defined by the 
words $w_1^{i+1},\dots,w_{k_{i+1}}^{i+1}$ and $L_{i+1}$ is 
defined by the words f $w_1^{i+1},\dots,w_{k_{i+1}}^{i+1}$.  
The words $w_1^{i+1},\dots,w_{k_{i+1}}^{i+1}$ are themselves concatenations of the words $u_1^i,\dots,u_{k_i}^i$ and induction hypothesis~\eqref{c5} guarantees that for any $1\leq j\leq k_i$ and any $1\leq t\leq k_{i+1}$, the relative frequency with which $u_j^i$ appears in the concatenation defining $w_t^{i+1}$ lies 
between $(1-\varepsilon_i)/k_i$ and $(1+\varepsilon_i)/k_i$.  Recall that no word of length $|w_{j_1}^i|$ occurs as a subword of both $w_{j_1}^iw_{j_1}^i$ and $w_{j_2}^iw_{j_2}^i$, for $j_1\neq j_2$, and so words of length $|w_{j_1}^i|$ that occur as subwords of $w_{j_1}^iw_{j_1}^i$ can occur only in $w_{j_1}^iw_{j_1}^i$ or possibly as a subword of $u_{j_2}^iu_{j_3}^i$ for some $j_2\neq j_3$ and in this case the occurrence must partially overlap both of the concatenated words.  Since $\mu$ is ergodic, if $w$ is a subword of some $w_j^iw_j^i$ of length $|w_j^i|$, then 
$$ 
\mu([w])=\frac{1}{2n+1}\sum_{m=-n}^n\one_{[w]}(\sigma^m x) 
$$ 
for $\mu$-almost every $x\in X_{\infty}$.  
Fix some such $x$ and choose some way to parse $x$ into a concatenation of the words $u_1^i,\dots,u_{k_i}^i$.  Let $\mathcal{I}\subseteq\Z$ be the locations where the words in this concatenation begin; this is an arithmetic progression in $\Z$ with gap $N_iM_i$.  The frequency with which a shift of $x$ brings one of the elements of $\mathcal{I}$ within distance $N_i$ of the origin (meaning when the word of length $N_i$ determined by this shift of $x$ is a word that partially overlaps the break between two of the words in our concatenation) is $N_i/N_iM_i=1/M_i$.  Thus we can check: 
\begin{align*}
\frac{(1-2\varepsilon_i)|s(w)|}{k_i|w_1^i|} &\leq\frac{(1-\varepsilon_i)|s(w)|}{k_i|w_1^i|}-\frac{N_i}{N_iM_i}\\
& \leq\lim_{n\to\infty}\frac{1}{2n+1}\sum_{m=-n}^n\one_{[w]}(\sigma^m x)\leq\frac{(1+\varepsilon_i)|s(w)|}{k_i|w_1^i|}+\frac{N_i}{N_iM_i}\\ 
& \leq\frac{(1+2\varepsilon_i)|s(w)|}{k_i|w_1^i|} 
\end{align*} 
since $M_i>|w_1^i|k_i/\varepsilon_i$ by induction hypothesis~\eqref{c6} (recall that $s(w)$ is $1$ unless it counts the number of occurrences of $w$ in the concatenation $w_j^iw_j^i$ where it occurs).
Therefore, 
$$ 
\frac{(1-2\varepsilon_i)|s(w)|}{k_i|w_1^i|}\leq\mu([w_j^i])\leq\frac{(1+2\varepsilon_i)|s(w)|}{k_i|w_1^i|} 
$$ 
for all $1\leq j\leq k_i$. 
Since 
$$ 
\mu\left(\bigcup_{w\in\mathcal{V}}[w]\right)\geq\frac{3}{4} 
$$ 
and since $\mu[w]\leq(1+2\varepsilon_i)/k_i$ for each $w\in\mathcal{V}$, 
it follows that  $|\mathcal{V}|\geq(3k_i)/(4+8\varepsilon_i)$. 

For each $1\leq j\leq k_i$, let $A_j\subseteq\{1,2,\dots,|w_j^i|\}$ be the set 
$$ 
A_j=\{s(w)\colon w\text{ is a word of length $|w_j^i|$ that occurs }w_j^iw_j^i\text{ and is in }\mathcal{V}\}. 
$$ 
Since 
$$ 
\mu\left(\bigcup_{w\in\mathcal{V}}[w]\right)\geq\frac{3}{4} 
$$ 
and since 
$$ 
\frac{(1-2\varepsilon_i)|s(w)|}{k_i|w_1^i|}\leq\mu[w]\leq\frac{(1+2\varepsilon_i)|s(w)|}{k_i|w_1^i|} 
$$ 
for all $w\in\mathcal{V}$, the number of elements of $\mathcal{V}$ is at least 
$$ 
\frac{3}{4}\cdot\frac{k_i|w_1^i|}{(1+2\varepsilon_i)}\geq\frac{k_i|w_1^i|}{2}, 
$$ 
where we count each $w\in\mathcal{V}$ with multiplicity $|s(w)|$.  Therefore for at least half of the integers, 
$1\leq j\leq k_i$ we have 
$$ 
|A_j|\geq\frac{k_i}{2}. 
$$ 
By Lemma~\ref{lem:combinatorics}, there exists some integer $1\leq s\leq k_i$ such that $s\in A_j$ for at least $1/4$ of the integers $1\leq j\leq k_i$.  Let $\mathcal{T}$ be this set of integers.  Then $\mathcal{V}$ contains the subword of length $|w_1^i|$ that occurs in $w_t^iw_t^i$, beginning at location $s$ for all $t\in\mathcal{T}$.  Since $d_H(w_{t_1}^i,w_{t_2}^i)>1/4$ for all $t_1\neq t_2$, 
it follows that the subword of length $w_{t_1}^i$ that occurs in $w_{t_1}^iw_{t_1}^i$ starting at location $s$ is Hamming distance at least $1/4$ from the analogous subword in $w_{t_2}^iw_{t_2}^i$.
It follows from our construction that every element of $\mathcal{V}$ is within Hamming distance $1/8$ of some element of $\mathcal{U}$.  Since two words of Hamming distance $1/4$ cannot be within distance $1/8$ of the same element of $\mathcal{U}$, it follows that $|\mathcal{U}|$ is at least $|\mathcal{T}|\geq k_i/4$.  But by construction, $k_i/4>b_{N_i}$, by induction hypothesis~\eqref{c2}.  Thus $|\mathcal{U}|>b_{N_i}$, and since $\mathcal{U}$ was arbitrary, ~\eqref{eq1} holds.

\subsubsection*{Analysis of $P_\sigma^-(n)$}  Our goal is to show that 
$$ 
K(P_i,\varepsilon_i,\sigma)<a_{P_i}. 
$$ 
Fix $i\in\N$.  Note that if $w\in\mathcal{L}_{P_i}(X_{\infty})$, then $w\in\mathcal{L}_{P_i}(X_{i+1})$.  By induction hypothesis~\eqref{c6} recall that if $\mathcal{W}$ is the set of words in the language of $X_{i+1}$ that occur as subwords of $v_1^{i+1},\dots,v_{k_{i+1}}^{i+1}$ then for any ergodic measure $\mu$ supported on $X_{i+1}$ we have 
$$ 
\mu\left(\bigcup_{w\in\mathcal{W}}[w]\right)>1-\varepsilon_i. 
$$ 
Therefore we can take $\varepsilon_i$-Hamming balls centered on words in $\mathcal{W}$
 as a way to cover a subset of $\mu$-measure at least $1-\varepsilon_i$.  
But by construction, the words $v_1^{i+1},\dots,v_{k_{i+1}}^{i+1}$ are all periodic words of period $|w_1^i|=|w_2^i|=\cdots=|w_{k_i}^i|$ and so 
$$ 
|\mathcal{W}|\leq|w_1^i|+|w_2^i|+\cdots+|w_{k_i}^i|. 
$$ 
By~\eqref{eq:def-of-P} we know that this quantity is at most $a_{P_i}$.  Therefore it is possible to cover a subset $X_{\infty}$ of $\mu$-measure at least $1-\varepsilon_i$ with at most $a_{P_i}$ many $\varepsilon_i$-Hamming balls around words of length $P_i$.  Therefore 
$$ 
K(P_i,\varepsilon_i,\sigma)<a_{P_i} 
$$ 
and so~\eqref{eq2} holds.

Our construction produces a subshift $(X_{\infty},\sigma)$ with the property that for any ergodic measure $\mu$ supported on $\sigma$, we have $P_{\sigma}^-(n)\prec(a_n)$ and $P_{\sigma}^+(n)\succ(b_n)$.  
Since all ergodic measures supported on $X_{\infty}$ satisfy $P_{\sigma}^-(n)\prec(a_n)$ and since $(a_n)$ grows subexponentially, Katok's theorem guarantees that $X_{\infty}$ supports only zero entropy measures.  
Furthermore, the Jewett-Krieger Theorem (see for example, Petersen~\cite{petersen}) guarantees that if $(Y,T,\mu)$ is an ergodic system of (measure theoretic) entropy less than $\log(N)$, then there is a minimal and uniquely ergodic 
system that is measure theoretically isomorphic to  our system.  Taking this model for the system, we have 
the existence of a subshift with all of the desired properties.  

This completes the proof of Theorem~\ref{thm:slow-entropy}. $ \hfill\square$

\end{document}